\documentclass[10pt]{article}
\usepackage{amsmath, amsthm, amssymb, amsfonts, mathrsfs, mathtools}
\usepackage{graphicx}
\usepackage{makeidx}
\usepackage[left=2cm,right=2cm,top=2cm,bottom=2cm]{geometry}
\usepackage{caption}
\usepackage{subcaption}
\usepackage{tikz}
\usetikzlibrary{decorations.pathreplacing}
\tikzstyle{vertex}=[auto=left,circle,draw=black,fill=white, inner sep=1.5]

\usepackage{hyperref}
\hypersetup{colorlinks=true, citecolor={blue}, linkcolor=blue, filecolor=magenta, urlcolor=cyan}

\newtheorem{theorem}{Theorem}[section]

\newtheorem{lem}[theorem]{Lemma}

\newtheorem{definition}{Definition}

\title{Some Bounds on the Double Domination of Signed Generalized Petersen Graphs and Signed I-Graphs}

\author{ Deepak Sehrawat\\
Department of Mathematics\\
Indian Institute of Technology Guwahati\\
Guwahati, India - 781039\\
Email: deepakmath55555@iitg.ac.in\\
\\
Bikash Bhattacharjya\\
Department of Mathematics\\
Indian Institute of Technology Guwahati\\
Guwahati, India - 781039\\
Email: b.bikash@iitg.ac.in
}
\date{}

\begin{document}
\maketitle

\vspace{-0.3in}

\vspace*{0.3in}
\noindent
\textbf{Abstract.} In a graph $G$, a vertex dominates itself and its neighbors. A subset
$D \subseteq V(G)$ is a double dominating set of $G$ if $D$ dominates every vertex of $G$ at least twice. A signed graph $\Sigma = (G,\sigma)$ is a graph $G$ together with an assignment $\sigma$ of positive or negative signs to all its edges. A cycle in a signed graph is positive if the product of its edge signs is positive. A signed graph is balanced if all its cycles are positive. A subset $D \subseteq V(\Sigma)$ is a double dominating set of $\Sigma$ if it satisfies the following conditions: (i) $D$ is a double dominating set of $G$, and (ii) $\Sigma[D:V \setminus D]$ is balanced, where $\Sigma[D:V \setminus D]$ is the subgraph of $\Sigma$ induced by the edges of $\Sigma$ with one end point in $D$ and the other end point in $V \setminus D$. The cardinality of a minimum double dominating set of $\Sigma$ is the double domination number $\gamma_{\times 2}(\Sigma)$. In this paper, we give bounds for the double domination number of signed cubic graphs. We also obtain some bounds on the double domination number of signed generalized Petersen graphs and signed I-graphs.

\vspace*{0.3in}
\noindent
\textbf{Keywords.} balance, switching, double domination, signed graph, generalized Petersen graph, I-graph.

\section{Introduction}
We consider only finite and simple graphs. For all the graph theoretic terms which are used in this paper but not defined, we refer the reader to~\cite{Bondy}. 

Let $G=(V,E)$ be a graph. We denote by $|V(G)|$ and $|E(G)|$ the size of the vertex set and the edge set of $G$, respectively. The \textit{open neighborhood} of a vertex $v \in V$ is $N(v) = \{u \in V~|~uv \in E \}$ and the \textit{closed neighborhood} is $N[v]= N(v) \cup \{v\}$. A vertex $v$ is said to dominate itself and its neighbors. A subset $D \subseteq V$ is said to be a \textit{dominating set} (DS) of $G$ if every vertex of $G$ is dominated by $D$. Equivalently, a subset $D$ of vertices is a \textit{dominating set} of $G$ if for each $v\in V$, $|N[v] \cap D| \geq 1$. The minimum cardinality of a dominating set is called the \textit{domination number} (DN) of $G$ and it is denoted by $\gamma(G)$. Different types of domination have been researched extensively. The literature on the studies of domination has been surveyed and detailed in the books ~\cite{Haynes} and \cite{Haynes2}.

In~\cite{Harary3}, Harary and Haynes defined a generalization of domination as follows: a subset $D \subseteq V$ is a \textit{k-tuple dominating} set of $G$ if for every vertex $v \in V$, either $v$ is in $D$ and has at least $k-1$ neighbors in $D$ or $v$ is in $V\setminus D$ and has at least $k$ neighbors in $D$. Equivalently, a subset $D \subseteq V$ is a \textit{k-tuple dominating} set of $G$ if for each $v\in V$, $|N[v] \cap D| \geq k$. A $2\text{-tuple}$ dominating set is called a \textit{double dominating set} (DDS). The cardinality of a minimum DDS of $G$ is called the \textit{double domination number} (DDN) of $G$. Some bounds for the double domination number in graphs are given in \cite{Hajian} and \cite{Harant}.

In~\cite{Harary2}, Harary introduced the notion of signed graphs and balance. A \textit{signed graph} is a graph whose edges are labelled with positive or negative signs. We denote it by $\Sigma = (G,\sigma)$, where $G$ is called the underlying graph of $\Sigma$ and $\sigma$ is called the \textit{signature~(signing)} of $G$. A signature $\sigma$ can also be viewed as a function from $E(G)$ into $\{+, -\}$. If the edges of $\Sigma$ are all positive, \textit{i.e.,} $\sigma^{-1}(-) = \emptyset$, then the signed graph is called the \textit{all positive} signed graph, and we denote it by $|\Sigma|$.

In a signed graph, \textit{switching} a vertex $v$ is to change the sign of each edge incident to $v$. If we switch every vertex of a subset $X$ of vertices, then we write the resulting signed graph as $\Sigma^{X}$. We say a signature $\Sigma_{1}$ is \textit{switching equivalent} or simply \textit{equivalent} to a signature $\Sigma_{2}$, denoted by $\Sigma_{1} \sim \Sigma_{2}$, if both $\Sigma_{1}$ and $\Sigma_{2}$ have the same underlying graph $G$ and $\Sigma_{1} = \Sigma_{2}^{X}$ for some $X \subseteq V(G)$.

A cycle in a signed graph is called \emph{positive} if the product of signs of its edges is positive, and negative, otherwise. A signed graph is \textit{balanced} if each of its cycles is balanced. The following theorem gives a necessary and sufficient condition for two signed graphs $\Sigma_{1}$ and $\Sigma_{2}$ to be switching equivalent.

\begin{theorem}~\cite{Zaslavsky}
\label{Signature}
Two signed graphs $\Sigma_{1}$ and $\Sigma_{2}$ are switching equivalent if and only if they have the same set of negative cycles.
\end{theorem}

Let $X \subseteq V(G)$ and $Y = V(G) \setminus X$. We denote by $[X:Y]$ the set of edges of $G$ with one end point in $X$ and the other end point in $Y$, and $|[X:Y]|$ denotes the number of edges in $[X:Y]$. The set $[X:Y]$ is called the \textit{edge cut} of $G$ associated with $X$. Further, $G[X:Y]$ denotes the subgraph of $G$ induced by the edges of $[X:Y]$. Similarly for $\Sigma = (G,\sigma)$, we denote by $\Sigma [X:Y]$ the subgraph induced by the edges of $\Sigma$ with one end point in $X$ and the other end point in $Y$.

Several notions of graph theory, such as the theory of nowhere zero flows and the theory of minors and graph homomorphisms, have been already extended to signed graphs. In 2013, Acharya~\cite{Acharya} extended the concept of domination to signed graphs. In 2016, Ashraf and Germina~\cite{Ashraf} generalized the notion of double domination to signed graphs as follows.

\begin{definition}\label{def_dd}~\cite{Ashraf}
\rm{A subset $D \subseteq V$ is a \textit{double dominating set} of  a signed graph $\Sigma$ if it satisfies the following two conditions: (i) for every $v \in V,~~|N[v] \cap D| \geq 2$ and, (ii) $\Sigma [D : V\setminus D]$ is balanced.}
\end{definition}

Clearly, Definition~\ref{def_dd} takes care of the concept of double domination in unsigned graphs. The cardinality of a minimum DDS of $\Sigma$ is called the double domination number of $\Sigma$ and is denoted by $\gamma_{\times 2}(\Sigma)$. The following theorem shows that the double domination is switching invariant. 

\begin{theorem}\cite{Ashraf}
Double domination is invariant under switching.
\end{theorem}


\begin{definition}\label{def_GPG}
\rm{For positive integers $n$ and $k$ satisfying $2 \leq 2k <n$, the \textit{generalised Petersen graph} $P_{n,k}$ is defined by $$V(P_{n,k}) = \{u_{0},u_{1},...,u_{n-1},v_{0},v_{1},...,v_{n-1}\}~ \text{and} ~E(P_{n,k}) = \{u_{i}u_{i+1},u_{i}v_{i},v_{i}v_{i+k}~|~i=0,1,...,n-1\},$$ where the subscripts are read modulo $n$.}
\end{definition}

We denote the sets $\{u_{0},u_{1},...,u_{n-1}\}$ and $\{v_{0},v_{1},...,v_{n-1}\}$ by $U$ and $V_{v}$, respectively. From the definition, it is clear that $P_{n,k}$ is a cubic graph and $P_{5,2}$ is the well-known Petersen graph. The edges $u_{i}v_{i}$ for $0 \leq i\leq n-1$ are called the \textit{spokes} and we denote the set of spokes by $S_{s}$. The cycle induced by vertices of $U$ is called the \textit{outer cycle} of $P_{n,k}$ and is denoted by $C_{o}$. The cycle(s) induced by vertices of $V_{v}$ is(are) called the \textit{inner cycle(s)} of $P_{n,k}$. If $\text{gcd}(n,k) = d$ then the subgraph induced by vertices of $V_{v}$, consists of $d$ pairwise disjoint $\frac{n}{d}$-cycles. If $d >1$ then no two vertices among $v_{0}, v_{1},..., v_{d-1}$ can be in the same $\frac{n}{d}$-cycle.

\begin{definition}
\rm{The I-graph $I(n,j,k)$ is a graph with vertex set $$V(I(n,j,k)) = \{u_{0},u_{1},...,u_{n-1},v_{0},v_{1},...,v_{n-1}\}$$ and edge set $$E(I(n,j,k)) = \{u_{i}u_{i+j},u_{i}v_{i},v_{i}v_{i+k}~|~i=0,1,...,n-1\},$$ where subscripts are read modulo $n$.}
 \end{definition}
 
 The class of generalized Petersen graphs is a sub-class of the class of I-graphs.

 In \cite{Boben}, Boben, Pisanski and Zitnik have studied various properties of I-graphs such as connectedness, girth, and whether they are bipartite or vertex-transitive. They also characterized the automorphism groups of I-graphs. 

The rest of this paper is organized as follows. First we give a lower bound and an upper bound for the DDN of signed cubic graphs. Further, we show that if $D$ is a DDS of a cubic graph $G$ such that $|D|= \frac{|V(G)|}{2}$ then $G[D : V\setminus D]$ admits a cycle decomposition. Also we show that if $D$ with $|D|= \frac{|V(G)|}{2}$ is not a DDS of a cubic graph $G$ then it is not necessarily true that $G[D : V\setminus D]$ admits a cycle decomposition. Second we obtain some bounds on the DDN of signed generalized Petersen graphs. Finally, we give bounds on the DDN of signed I-graphs.

\section{Bounds on DDN of Signed Cubic Graph}

In \cite{Ashraf} the authors obtained a bound on the double domination number of a signed graph.

\begin{theorem}\cite{Ashraf}\label{bound_0}
Let $\Sigma$ be any signed graph without isolated vertices on $n$ vertices, then $2\leq \gamma_{\times 2}(\Sigma) \leq n$. Moreover, these bounds are sharp.
\end{theorem}

In the following theorem, we show that the lower bound of Theorem~\ref{bound_0} can be improved if the underlying graph of $\Sigma$ is cubic.

\begin{theorem}\label{general bounds}
For $m\geq 2$, let $\Sigma$ be any signed cubic graph on $2m$ vertices. Then $$m \leq \gamma_{\times 2}(\Sigma) \leq 2m.$$
\end{theorem}
\begin{proof}
The upper bound follows from Theorem~\ref{bound_0}. 

To get the lower bound all we need is to show that $\Sigma$ cannot have a DDS of size $m-1$. Suppose on the contrary that there exists a DDS $D$ of $\Sigma$ such that $|D|=m-1$. It is clear that each vertex of $D$ is adjacent to at most two vertices of $V \setminus D$ because $D$ is a DDS. Thus
\begin{equation}\label{eq 1}
|[D:V \setminus D]| \leq 2m-2.
\end{equation}

Also, since $D$ is a DDS, every vertex of $V \setminus D$ is adjacent to at least two vertices of $D$. Thus
\begin{equation}\label{eq 2}
|[V \setminus D:D]| \geq 2m+2.
\end{equation}
But it is impossible for a set $D$ to satisfy  (\ref{eq 1}) and (\ref{eq 2}) simultaneously. Thus the set $D$ cannot be a DDS of $\Sigma$. This implies that any DDS of $G$ (hence of $\Sigma$) must be of size at least $m$. Therefore we have $\gamma_{\times 2}(\Sigma) \geq m$, and this completes the proof.
\end{proof}

Note that the lower bound of Theorem~\ref{general bounds} can be achieved. For that, let $G$ be a disjoint union of $m$ copies of $K_{4}$. Take $\Sigma = (G,\sigma)$ such that all the edges of $\Sigma$ are positive, that is $\sigma^{-1}(-) = \emptyset$. It is then clear that $\Sigma$ is a signed cubic graph on $4m$ vertices and $\gamma_{\times 2}(\Sigma) = 2m$.

As an application of DDS, we show that if $D$ with $|D|= \frac{|V(G)|}{2}$ is a DDS of a cubic graph $G$ then $G[D:V \setminus D]$ admits a cycle decomposition. 

\begin{lem}\label{2-regular lemma}
For $m\geq 2$, let $G$ be a cubic graph on $2m$ vertices. If $D$ is a DDS of $G$ such that $|D|=m$ then $G[D:V \setminus D]$ is a 2-regular subgraph of $G$.
\end{lem}
\begin{proof}
Given that $D$ is a DDS of $G$ such that $|D|=m$, where $G$ is a cubic graph on $2m$ vertices. We complete the proof by showing that each vertex of $D$ and $V \setminus D$ is adjacent to exactly two vertices of $V \setminus D$ and $D$, respectively. 

Suppose on the contrary that there are $r$ vertices of $D$ each of which are adjacent to at most one vertex of $V \setminus D$, where $1\leq r \leq m$. Since $D$ is a DDS, each vertex of the remaining $m-r$ vertices of $D$ is adjacent to exactly two vertices of $V \setminus D$. Therefore 
\begin{equation}\label{eq 3}
|[D:V \setminus D]| \leq 2(m-r)+r = 2m-r.
\end{equation}
On the other hand, each vertex of $V \setminus D$ is adjacent to at least two vertices of $D$ as $D$ is a DDS. So we have 
\begin{equation}\label{eq 4}
|[V \setminus D:D]| \geq 2m.
\end{equation}
As $|[D:V \setminus D]|=|[V \setminus D:D]|$, inequalities (\ref{eq 3}) and (\ref{eq 4}) cannot hold simultaneously. Thus every vertex of $D$ is adjacent to exactly two vertices of $V \setminus D$. Similarly it can be shown that every vertex of $V \setminus D$ is adjacent to exactly two vertices of $D$. Hence $G[D:V \setminus D]$ is a 2-regular subgraph of $G$, and the proof is complete.
\end{proof}

A graph in which each vertex has even degree is called an \textit{even graph}. Veblen's theorem (see Theorem 2.7,~\cite{ Bondy}) says that a graph admits a cycle decomposition if and only if it is even. The following theorem is a direct consequence of Veblen's theorem and Lemma~\ref{2-regular lemma}.

\begin{theorem}\label{cor_1}
For $m\geq 2$, let $G$ be a cubic graph on $2m$ vertices. If $D$ is a DDS of $G$ such that $|D|=m$ then $G[D:V \setminus D]$ admits a cycle decomposition.
\end{theorem}

If $D$ is not a DDS of a cubic graph $G$ such that $|D|= \frac{|V(G)|}{2}$, then it is not necessary that $G[D:V \setminus D]$ is the union of vertex disjoint cycles. For instance, let $G = P_{4,1}$ and $D =
\{u_{0},u_{1},u_{2},u_{3}\}$. It is easy to see that $G[D:V \setminus D]$ is not a union of vertex disjoint cycles.

\subsection{Bounds on $\gamma_{\times 2}(P_{n,k},\sigma)$}

A DDS of $(P_{n,k}, \sigma)$ need not be a DDS of $(P_{n,k}, \sigma')$, where $\sigma'$ is not equivalent to $\sigma$. For example, let $\Sigma = (P_{4,1}, \sigma)$, where $\sigma$ is a signature of $P_{4,1}$ for which the outer cycle $C_{o}$ and the inner cycle $C_{i}$ are positive. It is easy to check that the set $D = \{u_{0},v_{0},u_{2},v_{2}\}$, see Figure~\ref{P4}, forms a DDS of $\Sigma$. But if we take a signature $\sigma'$ for which $C_{o}$ and $C_{i}$ are negative, then $D = \{u_{0},v_{0},u_{2},v_{2}\}$ does not satisfy the condition (ii) of Definition~\ref{def_dd}.

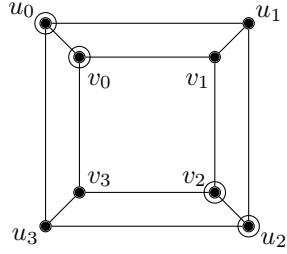
\begin{figure}[h]
\centering
\begin{tikzpicture}[scale=0.45]
\node[vertex] (v1) at (0,6) {};
\draw [black] (0,6) circle (9pt);
\node [below] at (-0.7,6.9) {$u_{0}$};
\node[vertex] (v2) at (6,6) {};
\node [below] at (6.6,6.8) {$u_{1}$};
\node[vertex] (v3) at (6,0) {};
\draw [black] (6,0) circle (9pt);
\node [below] at (6.75,0.1) {$u_{2}$};
\node[vertex] (v4) at (0,0) {};
\node [below] at (-0.6,0.2) {$u_{3}$};
\node[vertex] (v5) at (1,5) {};
\draw [black] (1,5) circle (9pt);
\node [below] at (1.6,4.8) {$v_{0}$};
\node[vertex] (v6) at (5,5) {};
\node [below] at (4.5,4.8) {$v_{1}$};
\node[vertex] (v7) at (5,1) {};
\draw [black] (5,1) circle (9pt);
\node [below] at (4.4,1.9) {$v_{2}$};
\node[vertex] (v8) at (1,1) {};
\node [below] at (1.6,1.9) {$v_{3}$};

\filldraw [black] (0,6) circle (3.5pt);
\filldraw [black] (6,6) circle (3.5pt);
\filldraw [black] (6,0) circle (3.5pt);
\filldraw [black] (0,0) circle (3.5pt);
\filldraw [black] (1,5) circle (3.5pt);
\filldraw [black] (5,5) circle (3.5pt);
\filldraw [black] (5,1) circle (3.5pt);
\filldraw [black] (1,1) circle (3.5pt);

\foreach \from/\to in {v1/v2,v2/v3,v3/v4,v4/v1,v5/v6,v6/v7,v7/v8,v8/v5,v1/v5,v2/v6,v3/v7,v4/v8} \draw (\from) -- (\to);

\end{tikzpicture}
\caption{A DDS of $P_{4,1}$.}
\label{P4}
\end{figure}

We saw that for two distinct signed graphs $\Sigma_1$ and $\Sigma_2$ having same underlying graph, a DDS of $\Sigma_1$ need not be a DDS of $\Sigma_2$. So in order to get upper bound of DDN of signed generalized Petersen graphs, we will construct DDS of generalized Petersen graphs in such a way that they satisfy condition (ii) of Definition~\ref{def_dd} for all possible signatures of generalized Petersen graphs. 

\begin{figure}[h]
\centering
\begin{tikzpicture}[scale=0.45]
\node[vertex] (v1) at (0,3.2) {};
\draw [black] (0,3.2) circle (9pt);
\node [below] at (0,4.5) {$u_{0}$};
\node[vertex] (v2) at (0,0) {};
\draw [black] (0,0) circle (9pt);
\node [below] at (0,-0.35) {$v_{0}$};
\node[vertex] (v3) at (3,3.2) {};
\node [below] at (3,4.5) {$u_{1}$};
\node[vertex] (v4) at (3,0) {};
\node [below] at (3,-0.35) {$v_{1}$};
\node[vertex] (v5) at (6,3.2) {};
\draw [black] (6,3.2) circle (9pt);
\node [below] at (6,4.5) {$u_{2}$};
\node[vertex] (v6) at (6,0) {};
\draw [black] (6,0) circle (9pt);
\node [below] at (6,-0.35) {$v_{2}$};
\node[vertex] (v7) at (9,3.2) {};
\node [below] at (9,4.5) {$u_{3}$};
\node[vertex] (v8) at (9,0) {};
\node [below] at (9,-0.35) {$v_{3}$};
\node[vertex] (v9) at (-3,3.2) {};
\draw [black] (-3,3.2) circle (9pt);
\node [below] at (-3,4.5) {$u_{2m}$};
\node[vertex] (v10) at (-3,0) {};
\node [below] at (-3,-0.35) {$v_{2m}$};
\node[vertex] (v11) at (-6,3.2) {};
\draw [black] (-6,3.2) circle (9pt);
\node [below] at (-6,4.5) {$u_{2m-1}$};
\node[vertex] (v12) at (-6,0) {};
\node [below] at (-6,-0.35) {$v_{2m-1}$};
\node[vertex] (v13) at (-9,3.2) {};
\draw [black] (-9,3.2) circle (9pt);
\node [below] at (-9,4.5) {$u_{2m-2}$};
\node[vertex] (v14) at (-9,0) {};
\draw [black] (-9,0) circle (9pt);
\node [below] at (-9,-0.35) {$v_{2m-2}$};
\node[vertex] (v15) at (-12,3.2) {};
\node [below] at (-12,4.5) {$u_{2m-3}$};
\node[vertex] (v16) at (-12,0) {};
\node [below] at (-12,-0.35) {$v_{2m-3}$};

\filldraw [black] (0,3.2) circle (3.5pt);
\filldraw [black] (3,3.2) circle (3.5pt);
\filldraw [black] (6,3.2) circle (3.5pt);
\filldraw [black] (9,3.2) circle (3.5pt);
\filldraw [black] (0,0) circle (3.5pt);\filldraw [black] (9,0) circle (3.5pt);
\filldraw [black] (3,0) circle (3.5pt);
\filldraw [black] (6,0) circle (3.5pt);
\filldraw [black] (-3,3.2) circle (3.5pt);
\filldraw [black] (-3,0) circle (3.5pt);
\filldraw [black] (-6,3.2) circle (3.5pt);
\filldraw [black] (-6,0) circle (3.5pt);
\filldraw [black] (-9,3.2) circle (3.5pt);
\filldraw [black] (-9,0) circle (3.5pt);
\filldraw [black] (-12,3.2) circle (3.5pt);
\filldraw [black] (-12,0) circle (3.5pt);

\foreach \from/\to in {v1/v3,v3/v5,v5/v7,v1/v2,v2/v4,v3/v4,v4/v6,v5/v6,v6/v8,v7/v8,v9/v1,v9/v10,v10/v12,v11/v12,v11/v13,v13/v14,v2/v10,v9/v11,v12/v14,v14/v16,v15/v16,v13/v15} \draw (\from) -- (\to);
\draw (9,3.2) -- (12,3.2);
\draw (9,0) -- (12,0);
\draw (-12,3.2) -- (-15,3.2);
\draw (-12,0) -- (-15,0);

\end{tikzpicture}
\caption{A DDS of $P_{2m+1,1}$.}
\label{P_2m+1}
\end{figure}
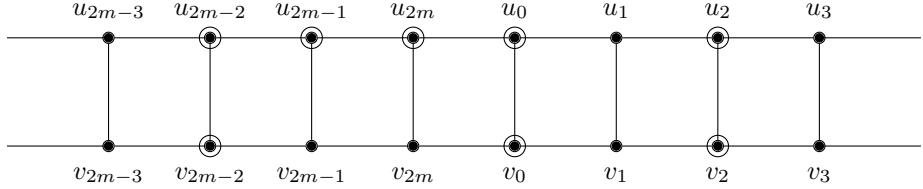

\begin{figure}[h]
\centering
\begin{tikzpicture}[scale=0.45]
\node[vertex] (v1) at (0,3.2) {};
\draw [black] (0,3.2) circle (9pt);
\node [below] at (0,4.5) {$u_{0}$};
\node[vertex] (v2) at (0,0) {};
\draw [black] (0,0) circle (9pt);
\node [below] at (0,-0.35) {$v_{0}$};
\node[vertex] (v3) at (3,3.2) {};
\node [below] at (3,4.5) {$u_{1}$};
\node[vertex] (v4) at (3,0) {};
\node [below] at (3,-0.35) {$v_{1}$};
\node[vertex] (v5) at (6,3.2) {};
\draw [black] (6,3.2) circle (9pt);
\node [below] at (6,4.5) {$u_{2}$};
\node[vertex] (v6) at (6,0) {};
\draw [black] (6,0) circle (9pt);
\node [below] at (6,-0.35) {$v_{2}$};
\node[vertex] (v7) at (9,3.2) {};
\node [below] at (9,4.5) {$u_{3}$};
\node[vertex] (v8) at (9,0) {};
\node [below] at (9,-0.35) {$v_{3}$};
\node[vertex] (v9) at (-3,3.2) {};
\draw [black] (-3,3.2) circle (9pt);
\node [below] at (-3,4.5) {$u_{2m-1}$};
\node[vertex] (v10) at (-3,0) {};
\draw [black] (-3,0) circle (9pt);
\node [below] at (-3,-0.35) {$v_{2m-1}$};
\node[vertex] (v11) at (-6,3.2) {};
\draw [black] (-6,3.2) circle (9pt);
\node [below] at (-6,4.5) {$u_{2m-2}$};
\node[vertex] (v12) at (-6,0) {};
\draw [black] (-6,0) circle (9pt);
\node [below] at (-6,-0.35) {$v_{2m-2}$};
\node[vertex] (v13) at (-9,3.2) {};
\node [below] at (-9,4.5) {$u_{2m-3}$};
\node[vertex] (v14) at (-9,0) {};
\node [below] at (-9,-0.35) {$v_{2m-3}$};

\filldraw [black] (0,3.2) circle (3.5pt);
\filldraw [black] (3,3.2) circle (3.5pt);
\filldraw [black] (6,3.2) circle (3.5pt);
\filldraw [black] (9,3.2) circle (3.5pt);
\filldraw [black] (0,0) circle (3.5pt);\filldraw [black] (9,0) circle (3.5pt);
\filldraw [black] (3,0) circle (3.5pt);
\filldraw [black] (6,0) circle (3.5pt);
\filldraw [black] (-3,3.2) circle (3.5pt);
\filldraw [black] (-3,0) circle (3.5pt);
\filldraw [black] (-6,3.2) circle (3.5pt);
\filldraw [black] (-6,0) circle (3.5pt);
\filldraw [black] (-9,3.2) circle (3.5pt);
\filldraw [black] (-9,0) circle (3.5pt);

\foreach \from/\to in {v1/v3,v3/v5,v5/v7,v1/v2,v2/v4,v3/v4,v4/v6,v5/v6,v6/v8,v7/v8,v9/v1,v9/v10,v10/v12,v11/v12,v11/v13,v13/v14,v2/v10,v9/v11,v12/v14} \draw (\from) -- (\to);
\draw (9,3.2) -- (12,3.2);
\draw (9,0) -- (12,0);
\draw (-9,3.2) -- (-12,3.2);
\draw (-9,0) -- (-12,0);

\end{tikzpicture}
\caption{A DDS of $P_{2m,1}$.}
\label{P_2m}
\end{figure}
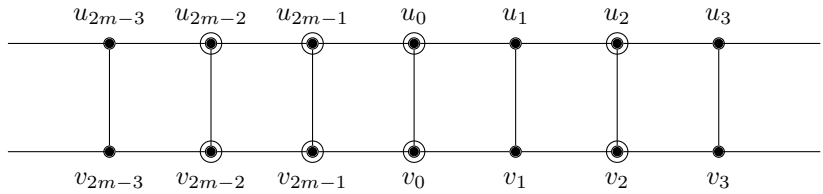

The following two lemmas will be used to get the bounds on the DDN of signed generalized Petersen graphs for $k=1$.

\begin{lem}\label{bound_1}
Let $\Sigma = (P_{2m+1,1}, \sigma)$ be any signed generalized Petersen graph. Then $$2m+1 \leq \gamma_{\times 2}(\Sigma) \leq 2m+2 .$$
\end{lem}
\begin{proof}
The lower bound follows from Theorem~\ref{general bounds}.

To get the upper bound, all we need is to construct a DDS of $\Sigma$ that uses $2m+2$ vertices. Consider the set $D=\{u_{2i},v_{2i}~|~i = 0,1,2...,m-1\} \cup \{u_{2m-1},u_{2m}\}$. It is easy to check that $D$ is a DDS of $P_{2m+1,1}$, as illustrated in Figure~\ref{P_2m+1}, and $|D| = 2m+2$. 

To complete the proof, it remains to show that $\Sigma [D:V\setminus D]$ is balanced. Note that \linebreak[4] $\Sigma[D:V\setminus D]$ is the union of two vertex disjoint paths $P_{1}~ \text{and}~ P_{2}$, where $P_{1}=u_{0}u_{1}u_{2}...u_{2m-3}u_{2m-2}$ and $P_{2}=u_{2m}v_{2m}v_{0}v_{1}v_{2}...v_{2m-2}v_{2m-1}u_{2m-1}$. This implies that $\Sigma[D:V \setminus D]$ is acyclic, and so $\Sigma [D:V \setminus D]$ is balanced. Thus $D$ is DDS of $\Sigma$.
\end{proof}

\begin{lem}\label{bound_2}
Let $\Sigma = (P_{2m,1}, \sigma)$ be any signed generalized Petersen graph. Then $$2m \leq \gamma_{\times 2}(\Sigma) \leq 2m+2.$$ Moreover, there exists a signed graph $\Sigma = (P_{2m,1}, \sigma)$, such that $\gamma_{\times 2}(\Sigma) = 2m$.
\end{lem}
\begin{proof}
From Theorem~\ref{general bounds}, it is obvious that $2m \leq \gamma_{\times 2}(\Sigma)$.
 
To get the upper bound, we need to produce a DDS of $\Sigma$ having $2m+2$ vertices. Consider the set $D=\{u_{2i},v_{2i}~|~i = 0,1,2...,m-1\} \cup \{u_{2m-1},v_{2m-1}\}$, as depicted in Figure~\ref{P_2m}. It is clear that each vertex of $P_{2m,1}$ is dominated at least twice by $D$, and that $|D| = 2m+2$.

Now we show that $\Sigma[D:V \setminus D]$ is balanced. Notice that $[D:V \setminus D] = E(P_{1}) \cup E(P_{2})$, where $P_{1}=u_{0}u_{1}u_{2}...u_{2m-2}$ and $P_{2}=v_{0}v_{1}v_{2}...v_{2m-2}$ are two vertex disjoint paths. Therefore $\Sigma[D:V \setminus D]$ is acyclic, and so $\Sigma[D:V \setminus D]$ is balanced. This shows that $D$ is a DDS of $\Sigma$. Hence $\gamma_{\times 2}(\Sigma) \leq 2m+2$.

 Let $\Sigma=(P_{2m,1}, \sigma)$, where $\sigma$ is any signature such that both the outer cycle $C_{o}$ and the inner cycle $C_{i}$ are positive in $\Sigma$. Consider the set $D=\{u_{2i},v_{2i}~|~i = 0,1,2...,m-1\}$, which is clearly a DDS of $P_{2m,1}$, and that $|D|=2m$. It is easy to see that $\Sigma[D:V \setminus D] = C_{0} \cup C_{i}$. Thus $\Sigma[D:V \setminus D]$ is balanced as $C_{o}~ \text{and}~C_{i}$ are positive in $\Sigma$. Hence $\gamma_{\times 2}(\Sigma) = 2m$, and the proof is complete.
\end{proof}

Lemma~\ref{bound_1} and Lemma~\ref{bound_2} together yield the following theorem. 

\begin{theorem}\label{bound_(n,1)}
Let $\Sigma = (P_{n,1}, \sigma)$ be any signed generalized Petersen graph. Then $$ n \leq \gamma_{\times 2}(\Sigma) \leq 2\big(\lfloor \frac{n}{2} \rfloor +1\big).  $$ 
\end{theorem}

\begin{figure}[h]
\begin{center}
\begin{tikzpicture}[scale=0.55]

		\node[vertex] (v5) at ({360/17*0 }:7cm) {};
		\node[vertex] (v4) at ({360/17*1 }:7cm) {};
		\node[vertex] (v3) at ({360/17 *2 }:7cm) {};
		\node[vertex] (v2) at ({360/17 *3 }:7cm) {};
		\node[vertex] (v1) at ({360/17 *4 }:7cm) {};
		\node[vertex] (v17) at ({360/17 *5}:7cm) {};
		\node[vertex] (v16) at ({360/17 *6 }:7cm) {};
		\node[vertex] (v15) at ({360/17 *7}:7cm) {};
		\node[vertex] (v14) at ({360/17 *8}:7cm) {};
		\node[vertex] (v13) at ({360/17 *9}:7cm) {};
		\node[vertex] (v12) at ({360/17 *10}:7cm) {};
		\node[vertex] (v11) at ({360/17 *11}:7cm) {};
        \node[vertex] (v10) at ({360/17 *12}:7cm) {};
        \node[vertex] (v9) at ({360/17 *13}:7cm) {};
        \node[vertex] (v8) at ({360/17 *14}:7cm) {};
        \node[vertex] (v7) at ({360/17 *15}:7cm) {};
        \node[vertex] (v6) at ({360/17 *16}:7cm) {};
        \draw ({360/17 *0 +1.5}:7cm) arc  ({360/17*0+1.5}:{360/17*1 -1.5 }:7cm);
        \draw ({360/17 *1 +1.5}:7cm) arc  ({360/17*1+1.5}:{360/17*2 -1.5 }:7cm);
        \draw ({360/17 *2 +1.5}:7cm) arc  ({360/17*2+1.5}:{360/17*3 -1.5 }:7cm);
        \draw ({360/17 *3 +1.5}:7cm) arc  ({360/17*3+1.5}:{360/17*4 -1.5 }:7cm);
        \draw ({360/17 *4 +1.5}:7cm) arc  ({360/17*4+1.5}:{360/17*5 -1.5 }:7cm);
        \draw ({360/17 *5 +1.5}:7cm) arc  ({360/17*5+1.5}:{360/17*6 -1.5 }:7cm);
        \draw ({360/17 *6 +1.5}:7cm) arc  ({360/17*6+1.5}:{360/17*7 -1.5 }:7cm);
        \draw ({360/17 *7 +1.5}:7cm) arc  ({360/17*7+1.5}:{360/17*8 -1.5 }:7cm);
        \draw ({360/17 *8 +1.5}:7cm) arc  ({360/17*8+1.5}:{360/17*9 -1.5 }:7cm);
        \draw ({360/17 *9 +1.5}:7cm) arc  ({360/17*9+1.5}:{360/17*10 -1.5 }:7cm);
        \draw ({360/17 *10 +1.5}:7cm) arc  ({360/17*10+1.5}:{360/17*11 -1.5 }:7cm);
        \draw ({360/17 *11+1.5}:7cm) arc  ({360/17*11+1.5}:{360/17*12 -1.5 }:7cm);
        \draw ({360/17 *12 +1.5}:7cm) arc  ({360/17*12+1.5}:{360/17*13 -1.5 }:7cm);
        \draw ({360/17 *13 +1.5}:7cm) arc  ({360/17*13+1.5}:{360/17*14 -1.5 }:7cm);
        \draw ({360/17 *14 +1.5}:7cm) arc  ({360/17*14+1.5}:{360/17*15 -1.5 }:7cm);
        \draw ({360/17 *15 +1.5}:7cm) arc  ({360/17*15+1.5}:{360/17*16 -1.5 }:7cm);
        \draw ({360/17 *16 +1.5}:7cm) arc  ({360/17*16+1.5}:{360/17*17 -1.5 }:7cm);
        
        \node[xshift=0.4cm, yshift=-0.05cm] at ({360/17*0 }:7cm) {$u_{4}$};
		\node[xshift=0.4cm] at ({360/17*1 }:7cm) {$u_{3}$};
		\node[xshift=0.25cm, yshift=0.2cm] at ({360/17 *2 }:7cm) {$u_{2}$};
		\node[xshift=0.15cm,yshift=0.25cm] at ({360/17 *3 }:7cm) {$u_{1}$};
		\node[ yshift=0.3cm] at ({360/17 *4 }:7cm) {$u_{0}$};
		\node[xshift=-0.15cm, yshift=0.3cm] at ({360/17 *5}:7cm) {$u_{16}$};
		\node[xshift=-0.3cm, yshift=0.25cm] at ({360/17 *6 }:7cm) {$u_{15}$};
		\node[xshift=-0.4cm, yshift=0.25cm] at ({360/17 *7}:7cm) {$u_{14}$};
		\node[xshift=-0.45cm] at ({360/17 *8}:7cm) {$u_{13}$};
		\node[xshift=-0.45cm,  yshift=-0.15cm] at ({360/17 *9}:7cm) {$u_{12}$};
		\node[xshift=-0.3cm, yshift=-0.25cm] at ({360/17 *10}:7cm) {$u_{11}$};
		\node[xshift=-0.1cm, yshift=-0.35cm] at ({360/17 *11}:7cm) {$u_{10}$};
        \node[yshift=-0.35cm] at ({360/17 *12}:7cm) {$u_{9}$};
        \node[xshift=0.1cm, yshift=-0.35cm] at ({360/17 *13}:7cm) {$u_{8}$};
  \node[xshift=0.2cm, yshift=-0.3cm] at ({360/17 *14}:7cm) {$u_{7}$};
  \node[xshift=0.35cm, yshift=-0.25cm] at ({360/17 *15}:7cm) {$u_{6}$};
  \node[xshift=0.35cm, yshift=-0.15cm] at ({360/17 *16}:7cm) {$u_{5}$};
  
  \filldraw[black] ({360/17*0 }:7cm) circle (3pt);
  \filldraw [black] ({360/17 *1}:7cm) circle (3pt);
  \filldraw [black] ({360/17 *2}:7cm) circle (3pt);
  \filldraw [black] ({360/17 *3}:7cm) circle (3pt);
  \filldraw [black] ({360/17 *4}:7cm) circle (3pt);
  \filldraw [black] ({360/17 *5}:7cm) circle (3pt);
  \filldraw [black] ({360/17 *6}:7cm) circle (3pt);
  \filldraw [black] ({360/17 *7}:7cm) circle (3pt);
  \filldraw [black] ({360/17 *8}:7cm) circle (3pt);
  \filldraw [black] ({360/17 *9}:7cm) circle (3pt);
  \filldraw [black] ({360/17 *10}:7cm) circle (3pt);
  \filldraw [black] ({360/17 *11}:7cm) circle (3pt);
  \filldraw [black] ({360/17 *12}:7cm) circle (3pt);
  \filldraw [black] ({360/17 *13}:7cm) circle (3pt);
  \filldraw [black] ({360/17 *14}:7cm) circle (3pt);
  \filldraw [black] ({360/17 *15}:7cm) circle (3pt); 
  \filldraw [black] ({360/17 *16}:7cm) circle (3pt); 
  
  \draw [black] ({360/17*0 }:7cm) circle (7pt);
  \draw [black] ({360/17 *1}:7cm) circle (7pt);
  \draw [black] ({360/17 *2}:7cm) circle (7pt);
  \draw [black] ({360/17 *3}:7cm) circle (7pt);
  \draw [black] ({360/17 *4}:7cm) circle (7pt);
  \draw [black] ({360/17 *5}:7cm) circle (7pt);
  \draw [black] ({360/17 *6}:7cm) circle (7pt);
  \draw [black] ({360/17 *7}:7cm) circle (7pt);
  \draw [black] ({360/17 *8}:7cm) circle (7pt);
  \draw [black] ({360/17 *9}:7cm) circle (7pt);
  \draw [black] ({360/17 *10}:7cm) circle (7pt);
  \draw [black] ({360/17 *11}:7cm) circle (7pt);
  \draw [black] ({360/17 *12}:7cm) circle (7pt);
  \draw [black] ({360/17 *13}:7cm) circle (7pt);
  \draw [black] ({360/17 *14}:7cm) circle (7pt);
  \draw [black] ({360/17 *15}:7cm) circle (7pt);
  \draw [black] ({360/17 *16}:7cm) circle (7pt);
    
		\node[vertex] (v22) at ({360/17*0 }:5.5cm) {};
		\node[vertex] (v21) at ({360/17*1 }:5.5cm) {};
		\node[vertex] (v20) at ({360/17 *2 }:5.5cm) {};
		\node[vertex] (v19) at ({360/17 *3 }:5.5cm) {};
		\node[vertex] (v18) at ({360/17 *4 }:5.5cm) {};
		\node[vertex] (v34) at ({360/17 *5}:5.5cm) {};
		\node[vertex] (v33) at ({360/17 *6 }:5.5cm) {};
		\node[vertex] (v32) at ({360/17 *7}:5.5cm) {};
		\node[vertex] (v31) at ({360/17 *8}:5.5cm) {};
		\node[vertex] (v30) at ({360/17 *9}:5.5cm) {};
		\node[vertex] (v29) at ({360/17 *10}:5.5cm) {};
		\node[vertex] (v28) at ({360/17 *11}:5.5cm) {};
        \node[vertex] (v27) at ({360/17 *12}:5.5cm) {};
        \node[vertex] (v26) at ({360/17 *13}:5.5cm) {};
        \node[vertex] (v25) at ({360/17 *14}:5.5cm) {};
        \node[vertex] (v24) at ({360/17 *15}:5.5cm) {};
        \node[vertex] (v23) at ({360/17 *16}:5.5cm) {};
        
       \filldraw[black] ({360/17*0 }:5.5cm) circle (3pt);
  \filldraw [black] ({360/17 *1}:5.5cm) circle (3pt);
  \filldraw [black] ({360/17 *2}:5.5cm) circle (3pt);
  \filldraw [black] ({360/17 *3}:5.5cm) circle (3pt);
  \filldraw [black] ({360/17 *4}:5.5cm) circle (3pt);
  \filldraw [black] ({360/17 *5}:5.5cm) circle (3pt);
  \filldraw [black] ({360/17 *6}:5.5cm) circle (3pt);
  \filldraw [black] ({360/17 *7}:5.5cm) circle (3pt);
  \filldraw [black] ({360/17 *8}:5.5cm) circle (3pt);
  \filldraw [black] ({360/17 *9}:5.5cm) circle (3pt);
  \filldraw [black] ({360/17 *10}:5.5cm) circle (3pt);
  \filldraw [black] ({360/17 *11}:5.5cm) circle (3pt);
  \filldraw [black] ({360/17 *12}:5.5cm) circle (3pt);
  \filldraw [black] ({360/17 *13}:5.5cm) circle (3pt);
  \filldraw [black] ({360/17 *14}:5.5cm) circle (3pt);
  \filldraw [black] ({360/17 *15}:5.5cm) circle (3pt); 
  \filldraw [black] ({360/17 *16}:5.5cm) circle (3pt);         
        
        \draw ({360/17 *0 +1.5}:5.5cm) ;
        \draw ({360/17 *1 +1.5}:5.5cm) ;
        \draw ({360/17 *2 +1.5}:5.5cm) ;
        \draw ({360/17 *3 +1.5}:5.5cm) ;
        \draw ({360/17 *4 +1.5}:5.5cm) ;
        \draw ({360/17 *5 +1.5}:5.5cm) ;
        \draw ({360/17 *6 +1.5}:5.5cm) ;
        \draw ({360/17 *7 +1.5}:5.5cm) ;
        \draw ({360/17 *8 +1.5}:5.5cm) ;
        \draw ({360/17 *9 +1.5}:5.5cm) ;
        \draw ({360/17 *10 +1.5}:5.5cm) ;
        \draw ({360/17 *11+1.5}:5.5cm) ;
        \draw ({360/17 *12 +1.5}:5.5cm) ;
        \draw ({360/17 *13 +1.5}:5.5cm) ;
        \draw ({360/17 *14 +1.5}:5.5cm) ;
        \draw ({360/17 *15 +1.5}:5.5cm) ;
        \draw ({360/17 *16 +1.5}:5.5cm) ;
        
   \draw [black] ({360/17 *1}:5.5cm) circle (7pt);
  \draw [black] ({360/17 *2}:5.5cm) circle (7pt);
  \draw [black] ({360/17 *15}:5.5cm) circle (7pt);  
  \draw [black] ({360/17 *14}:5.5cm) circle (7pt); 
  \draw [black] ({360/17 *11}:5.5cm) circle (7pt);  
  \draw [black] ({360/17 *10}:5.5cm) circle (7pt); 
  \draw [black] ({360/17 *7}:5.5cm) circle (7pt);  
  \draw [black] ({360/17 *6}:5.5cm) circle (7pt);     
        
      \node[xshift=0.15cm, yshift=0.25cm] at ({360/17*0 }:5.5cm) {$v_{4}$};
		\node[xshift=0.05cm, yshift=0.25cm] at ({360/17*1 }:5.5cm) {$v_{3}$};
		\node[xshift=-0.05cm, yshift=0.25cm] at ({360/17 *2 }:5.5cm) {$v_{2}$};
		\node[xshift=-0.2cm, yshift=0.25cm] at ({360/17 *3 }:5.5cm) {$v_{1}$};
		\node[ xshift=-0.3cm, yshift=0.2cm] at ({360/17 *4 }:5.5cm) {$v_{0}$};
		\node[xshift=-0.35cm, yshift=0.1cm] at ({360/17 *5}:5.5cm) {$v_{16}$};
		\node[xshift=-0.4cm, yshift=0.05cm] at ({360/17 *6 }:5.5cm) {$v_{15}$};
		\node[xshift=-0.35cm, yshift=-0.15cm] at ({360/17 *7}:5.5cm) {$v_{14}$};
		\node[xshift=-0.25cm, yshift=-0.3cm] at ({360/17 *8}:5.5cm) {$v_{13}$};
		\node[xshift=-0.15cm, yshift=-0.35cm] at ({360/17 *9}:5.5cm) {$v_{12}$};
		\node[xshift=0.1cm, yshift=-0.35cm] at ({360/17 *10}:5.5cm) {$v_{11}$};
		\node[xshift=0.2cm, yshift=-0.35cm] at ({360/17 *11}:5.5cm) {$v_{10}$};
        \node[xshift=0.25cm, yshift=-0.25cm] at ({360/17 *12}:5.5cm) {$v_{9}$};
       \node[xshift=0.35cm, yshift=-0.15cm] at ({360/17 *13}:5.5cm) {$v_{8}$};
  \node[xshift=0.35cm, yshift=-0.1cm] at ({360/17 *14}:5.5cm) {$v_{7}$};
  \node[xshift=0.4cm, yshift=-0.05cm] at ({360/17 *15}:5.5cm) {$v_{6}$};
  \node[xshift=0.3cm, yshift=0.1cm] at ({360/17 *16}:5.5cm) {$v_{5}$};    

       \foreach \from/\to in {v1/v18,v2/v19,v3/v20,v4/v21,v5/v22,v6/v23,v7/v24,v8/v25,v9/v26,v10/v27,v11/v28,v12/v29,v13/v30,v14/v31,v15/v32,v16/v33,v17/v34,v18/v20,v19/v21,v20/v22,v21/v23,v22/v24,v23/v25,v24/v26,v25/v27,v26/v28,v27/v29,v28/v30,v29/v31,v30/v32,v31/v33,v32/v34,v33/v18,v34/v19} \draw (\from) -- (\to);
        
	\end{tikzpicture}
\caption{An example for the upper bound of Lemma~\ref{bound_3}: a DDS of $P_{17,2}$.}\label{P_17,2}
\end{center}
\end{figure}
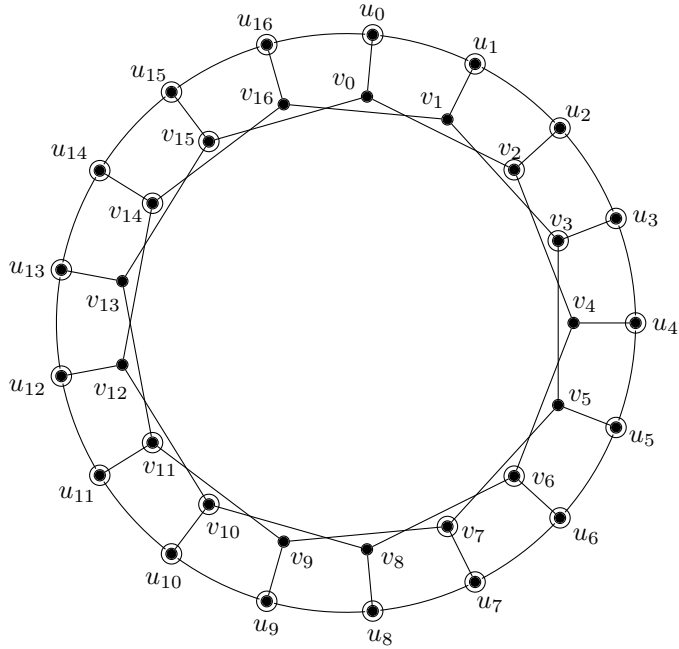

We will use the following two lemmas to get the bounds on the DDN of signed generalized Petersen graph $(P_{n,k}, \sigma)$, where $\gcd(n,k) = 1~\text{and}~k\geq 2$.

\begin{lem}\label{bound_3}
Let $\Sigma = (P_{n,k}, \sigma)$ be any signed generalized Petersen graph, where $\gcd(n,k) =1~\text{and}~k\geq2$. Let $\lceil \frac{n}{k} \rceil = 2m+1$, for some $m\geq1$. Then $ n \leq \gamma_{\times 2}(\Sigma) \leq n+mk$.
\end{lem}
\begin{proof}
Recall that $U$ denotes the set of $u\text{-vertices}$ and $V_v$ denotes the set of $v\text{-vertices}$ of $P_{n,k}$. Clearly $|U| = |V_{v}| = n$. Let $V_{1},V_{2},...,V_{2m},V_{2m+1}$ be a partition of the set $V_v$ such that \linebreak[4] $V_{i}= \{v_{(i-1)k},v_{(i-1)k+1},...v_{(i-1)k+(k-1)}\}$ for $1\leq i\leq 2m$ and $V_{2m+1} = V_v- \cup_{i=1}^{2m}V_i$. For each $1 \leq i \leq 2m$, it is obvious that $|V_{i}| = k$. Thus we have $|V_{2m+1}| = n-2mk$.

To get the upper bound, we take the set $D= U \cup \left(\cup_{i=1}^{m} V_{2i}\right)$. For example, see Figure~\ref{P_17,2}. It is clear that $|D|=n+km$. As the cycle $C_o$ lies completely inside $G[D]$, where $G = P_{n,k}$, every $u\text{-vertex}$ is dominated at least twice by $D$. Also for each $v\text{-vertex}$, the corresponding neighbor $u\text{-vertex}$ is in $D$. We show that each vertex of $V_v$ is either in $D$ or adjacent to at least one $v\text{-vertex}$ in $D$. Clearly each vertex of $\cup_{i=1}^{m} V_{2i}$ is in $D$. Further for $1\leq i \leq m$, each vertex of $V_{2i-1}$ is adjacent to a vertex of $V_{2i}$ and $V_{2i} \subseteq D$. Thus each vertex of $\cup_{i=1}^{m} V_{2i-1}$ is adjacent to a $v\text{-vertex}$ in $D$. Also each vertex $V_{2m+1}$ is adjacent to a vertex of $V_{2m}$ and $V_{2m} \subseteq D$. Thus each vertex of $V_{2m+1}$ is also adjacent to a $v\text{-vertex}$ in $D$. This shows that $D$ is a DDS of $P_{n,k}$.

Now we show that $\Sigma[D:V \setminus D]$ is balanced. To prove this, it is enough to show that $\Sigma[D:V \setminus D]$ is acyclic. Note that the vertex $u_{i}$ is adjacent to at most one vertex of $V \setminus D$, since $C_{o}$ lies in $G[D]$. Thus if $\Sigma[D:V \setminus D]$ contains any cycle then all the vertices of that cycle must be $v\text{-vertices}$ only. But the graph $P_{n,k}$ has only one inner cycle, say $C_i$, induced by $v\text{-vertices}$ as $\gcd(n,k) =1$. Therefore if $\Sigma[D:V \setminus D]$ contains a cycle then that cycle must be the inner cycle $C_{i}$ itself. Note that the vertices $v_{k-1}~\text{and}~v_{n-1}$ are adjacent and both belong to $V \setminus D$. Thus $\Sigma[D:V \setminus D]$ cannot contain a cycle. Therefore $\Sigma [D:V \setminus D]$ is balanced. This implies that $\gamma_{\times 2}(\Sigma) \leq n+mk$.

 The lower bound follows from Theorem~\ref{general bounds}, and the proof is complete.
\end{proof}

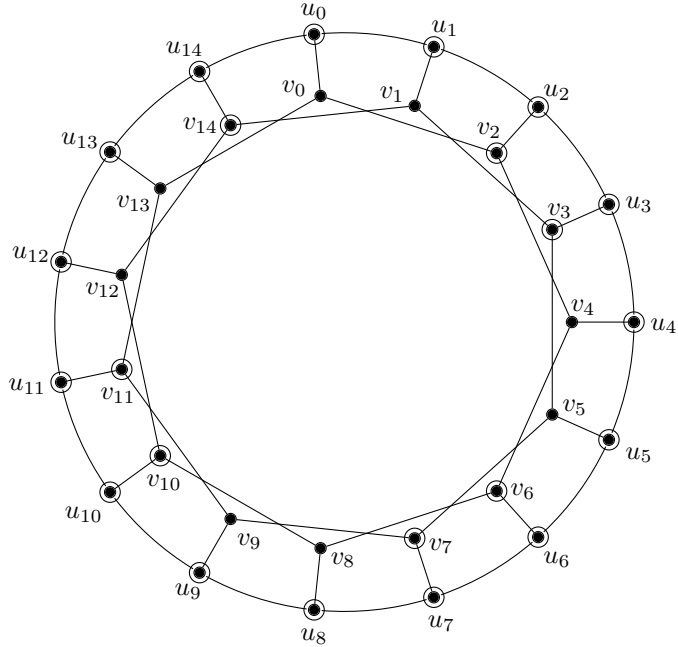
\begin{figure}[h]
\begin{center}
\begin{tikzpicture}[scale=0.55]

		\node[vertex] (v5) at ({360/15*0 }:7cm) {};
		\node[vertex] (v4) at ({360/15*1 }:7cm) {};
		\node[vertex] (v3) at ({360/15 *2 }:7cm) {};
		\node[vertex] (v2) at ({360/15 *3 }:7cm) {};
		\node[vertex] (v1) at ({360/15 *4 }:7cm) {};
		\node[vertex] (v15) at ({360/15 *5}:7cm) {};
		\node[vertex] (v14) at ({360/15 *6 }:7cm) {};
		\node[vertex] (v13) at ({360/15 *7}:7cm) {};
		\node[vertex] (v12) at ({360/15 *8}:7cm) {};
		\node[vertex] (v11) at ({360/15 *9}:7cm) {};
		\node[vertex] (v10) at ({360/15 *10}:7cm) {};
		\node[vertex] (v9) at ({360/15 *11}:7cm) {};
        \node[vertex] (v8) at ({360/15 *12}:7cm) {};
        \node[vertex] (v7) at ({360/15 *13}:7cm) {};
        \node[vertex] (v6) at ({360/15 *14}:7cm) {};
        \draw ({360/15 *0 +1.5}:7cm) arc  ({360/15*0+1.5}:{360/15*1 -1.5 }:7cm);
        \draw ({360/15 *1 +1.5}:7cm) arc  ({360/15*1+1.5}:{360/15*2 -1.5 }:7cm);
        \draw ({360/15 *2 +1.5}:7cm) arc  ({360/15*2+1.5}:{360/15*3 -1.5 }:7cm);
        \draw ({360/15 *3 +1.5}:7cm) arc  ({360/15*3+1.5}:{360/15*4 -1.5 }:7cm);
        \draw ({360/15 *4 +1.5}:7cm) arc  ({360/15*4+1.5}:{360/15*5 -1.5 }:7cm);
        \draw ({360/15*5 +1.5}:7cm) arc  ({360/15*5+1.5}:{360/15*6 -1.5 }:7cm);
        \draw ({360/15 *6 +1.5}:7cm) arc  ({360/15*6+1.5}:{360/15*7 -1.5 }:7cm);
        \draw ({360/15*7 +1.5}:7cm) arc  ({360/15*7+1.5}:{360/15*8 -1.5 }:7cm);
        \draw ({360/15 *8 +1.5}:7cm) arc  ({360/15*8+1.5}:{360/15*9 -1.5 }:7cm);
        \draw ({360/15 *9 +1.5}:7cm) arc  ({360/15*9+1.5}:{360/15*10 -1.5 }:7cm);
        \draw ({360/15 *10 +1.5}:7cm) arc  ({360/15*10+1.5}:{360/15*11 -1.5 }:7cm);
        \draw ({360/15 *11+1.5}:7cm) arc  ({360/15*11+1.5}:{360/15*12 -1.5 }:7cm);
        \draw ({360/15 *12 +1.5}:7cm) arc  ({360/15*12+1.5}:{360/15*13 -1.5 }:7cm);
        \draw ({360/15 *13 +1.5}:7cm) arc  ({360/15*13+1.5}:{360/15*14 -1.5 }:7cm);
        \draw ({360/15 *14 +1.5}:7cm) arc  ({360/15*14+1.5}:{360/15*15 -1.5 }:7cm);
        
        \node[xshift=0.4cm, yshift=-0.05cm] at ({360/15*0 }:7cm) {$u_{4}$};
		\node[xshift=0.4cm] at ({360/15*1 }:7cm) {$u_{3}$};
		\node[xshift=0.25cm, yshift=0.2cm] at ({360/15 *2 }:7cm) {$u_{2}$};
		\node[xshift=0.15cm,yshift=0.25cm] at ({360/15 *3 }:7cm) {$u_{1}$};
		\node[ yshift=0.3cm] at ({360/15 *4 }:7cm) {$u_{0}$};
		\node[xshift=-0.2cm, yshift=0.3cm] at ({360/15 *5}:7cm) {$u_{14}$};
		\node[xshift=-0.4cm, yshift=0.15cm] at ({360/15 *6 }:7cm) {$u_{13}$};
		\node[xshift=-0.4cm, yshift=0.1cm] at ({360/15 *7}:7cm) {$u_{12}$};
		\node[xshift=-0.45cm, yshift=-0.05cm] at ({360/15 *8}:7cm) {$u_{11}$};
		\node[xshift=-0.35cm,  yshift=-0.3cm] at ({360/15 *9}:7cm) {$u_{10}$};
		\node[xshift=-0.15cm, yshift=-0.25cm] at ({360/15 *10}:7cm) {$u_{9}$};
		\node[yshift=-0.35cm] at ({360/15 *11}:7cm) {$u_{8}$};
        \node[xshift=0.1cm, yshift=-0.35cm] at ({360/15 *12}:7cm) {$u_{7}$};
        \node[xshift=0.25cm, yshift=-0.3cm] at ({360/15 *13}:7cm) {$u_{6}$};
  \node[xshift=0.4cm, yshift=-0.2cm] at ({360/15 *14}:7cm) {$u_{5}$};
  
  \filldraw [black] ({360/15 *0}:7cm) circle (3pt);
  \filldraw [black] ({360/15 *1}:7cm) circle (3pt);
  \filldraw [black] ({360/15 *2}:7cm) circle (3pt);
  \filldraw [black] ({360/15 *3}:7cm) circle (3pt);
  \filldraw [black] ({360/15 *4}:7cm) circle (3pt);
  \filldraw [black] ({360/15 *5}:7cm) circle (3pt);
  \filldraw [black] ({360/15 *6}:7cm) circle (3pt);
  \filldraw [black] ({360/15 *7}:7cm) circle (3pt);
  \filldraw [black] ({360/15 *8}:7cm) circle (3pt);
  \filldraw [black] ({360/15 *9}:7cm) circle (3pt);
  \filldraw [black] ({360/15 *10}:7cm) circle (3pt);
  \filldraw [black] ({360/15 *11}:7cm) circle (3pt);
  \filldraw [black] ({360/15 *12}:7cm) circle (3pt);
  \filldraw [black] ({360/15 *13}:7cm) circle (3pt);
  \filldraw [black] ({360/15 *14}:7cm) circle (3pt);
  
  \draw [black] ({360/15*0 }:7cm) circle (7pt);
  \draw [black] ({360/15 *1}:7cm) circle (7pt);
  \draw [black] ({360/15 *2}:7cm) circle (7pt);
  \draw [black] ({360/15 *3}:7cm) circle (7pt);
  \draw [black] ({360/15 *4}:7cm) circle (7pt);
  \draw [black] ({360/15 *5}:7cm) circle (7pt);
  \draw [black] ({360/15 *6}:7cm) circle (7pt);
  \draw [black] ({360/15 *7}:7cm) circle (7pt);
  \draw [black] ({360/15 *8}:7cm) circle (7pt);
  \draw [black] ({360/15 *9}:7cm) circle (7pt);
  \draw [black] ({360/15 *10}:7cm) circle (7pt);
  \draw [black] ({360/15 *11}:7cm) circle (7pt);
  \draw [black] ({360/15 *12}:7cm) circle (7pt);
  \draw [black] ({360/15 *13}:7cm) circle (7pt);
  \draw [black] ({360/15 *14}:7cm) circle (7pt);
  
  \node[vertex] (v20) at ({360/15*0 }:5.5cm) {};
  \node[vertex] (v19) at ({360/15*1 }:5.5cm) {};
  \node[vertex] (v18) at ({360/15*2 }:5.5cm) {};
  \node[vertex] (v17) at ({360/15*3 }:5.5cm) {};
  \node[vertex] (v16) at ({360/15*4 }:5.5cm) {};
  \node[vertex] (v30) at ({360/15*5 }:5.5cm) {};
  \node[vertex] (v29) at ({360/15*6 }:5.5cm) {};
  \node[vertex] (v28) at ({360/15*7 }:5.5cm) {};
  \node[vertex] (v27) at ({360/15*8 }:5.5cm) {};
  \node[vertex] (v26) at ({360/15*9 }:5.5cm) {};
  \node[vertex] (v25) at ({360/15*10 }:5.5cm) {};
  \node[vertex] (v24) at ({360/15*11 }:5.5cm) {};
  \node[vertex] (v23) at ({360/15*12 }:5.5cm) {};
  \node[vertex] (v22) at ({360/15*13 }:5.5cm) {};
  \node[vertex] (v21) at ({360/15*14 }:5.5cm) {};
  
    \draw ({360/15 *0 +1.5}:5.5cm) ;
    \draw ({360/15 *1 +1.5}:5.5cm) ;
    \draw ({360/15 *2 +1.5}:5.5cm) ;
    \draw ({360/15 *3 +1.5}:5.5cm) ;
    \draw ({360/15 *4 +1.5}:5.5cm) ;
    \draw ({360/15 *5 +1.5}:5.5cm) ;
    \draw ({360/15 *6 +1.5}:5.5cm) ;
    \draw ({360/15 *7 +1.5}:5.5cm) ;
    \draw ({360/15 *8 +1.5}:5.5cm) ;
    \draw ({360/15 *9 +1.5}:5.5cm) ;
    \draw ({360/15 *10 +1.5}:5.5cm) ;
    \draw ({360/15 *11 +1.5}:5.5cm) ;
    \draw ({360/15 *12 +1.5}:5.5cm) ;
    \draw ({360/15 *13 +1.5}:5.5cm) ;
    \draw ({360/15 *14 +1.5}:5.5cm) ;
    
  \filldraw [black] ({360/15 *0}:5.5cm) circle (3pt);
  \filldraw [black] ({360/15 *1}:5.5cm) circle (3pt);
  \filldraw [black] ({360/15 *2}:5.5cm) circle (3pt);
  \filldraw [black] ({360/15 *3}:5.5cm) circle (3pt);
  \filldraw [black] ({360/15 *4}:5.5cm) circle (3pt);
  \filldraw [black] ({360/15 *5}:5.5cm) circle (3pt);
  \filldraw [black] ({360/15 *6}:5.5cm) circle (3pt);
  \filldraw [black] ({360/15 *7}:5.5cm) circle (3pt);
  \filldraw [black] ({360/15 *8}:5.5cm) circle (3pt);
  \filldraw [black] ({360/15 *9}:5.5cm) circle (3pt);
  \filldraw [black] ({360/15 *10}:5.5cm) circle (3pt);
  \filldraw [black] ({360/15 *11}:5.5cm) circle (3pt);
  \filldraw [black] ({360/15 *12}:5.5cm) circle (3pt);
  \filldraw [black] ({360/15 *13}:5.5cm) circle (3pt);
  \filldraw [black] ({360/15 *14}:5.5cm) circle (3pt);
  
  \draw [black] ({360/15*2 }:5.5cm) circle (7pt);
  \draw [black] ({360/15*1 }:5.5cm) circle (7pt);
  \draw [black] ({360/15*13 }:5.5cm) circle (7pt);
  \draw [black] ({360/15*12 }:5.5cm) circle (7pt);
  \draw [black] ({360/15*9 }:5.5cm) circle (7pt);
  \draw [black] ({360/15*8 }:5.5cm) circle (7pt);
  \draw [black] ({360/15*5 }:5.5cm) circle (7pt);
  
   \node[xshift=0.3cm, yshift=0.1cm] at ({360/15 *14}:5.5cm) {$v_{5}$};
   \node[xshift=0.35cm] at ({360/15 *13}:5.5cm) {$v_{6}$};
   \node[xshift=0.35cm, yshift=-0.1cm] at ({360/15 *12}:5.5cm) {$v_{7}$};
   \node[xshift=0.3cm, yshift=-0.155cm] at ({360/15 *11}:5.5cm) {$v_{8}$};
   \node[xshift=0.25cm, yshift=-0.3cm] at ({360/15 *10}:5.5cm) {$v_{9}$};
   \node[xshift=0.05cm, yshift=-0.35cm] at ({360/15 *9}:5.5cm) {$v_{10}$};
   \node[xshift=-0.05cm, yshift=-0.35cm] at ({360/15 *8}:5.5cm) {$v_{11}$};
   \node[xshift=-0.25cm, yshift=-0.2cm] at ({360/15 *7}:5.5cm) {$v_{12}$};
   \node[xshift=-0.35cm, yshift=-0.2cm] at ({360/15 *6}:5.5cm) {$v_{13}$};
   \node[xshift=-0.4cm] at ({360/15 *5}:5.5cm) {$v_{14}$};
   \node[xshift=-0.35cm, yshift=0.1cm] at ({360/15 *4}:5.5cm) {$v_{0}$};
   \node[xshift=-0.3cm, yshift=0.15cm] at ({360/15 *3}:5.5cm) {$v_{1}$};
   \node[xshift=-0.1cm, yshift=0.3cm] at ({360/15 *2}:5.5cm) {$v_{2}$};
   \node[xshift=0.1cm, yshift=0.25cm] at ({360/15 *1}:5.5cm) {$v_{3}$};
   \node[xshift=0.15cm, yshift=0.2cm] at ({360/15 *0}:5.5cm) {$v_{4}$};
  
  \foreach \from/\to in {v5/v20,v4/v19,v3/v18,v2/v17,v1/v16,v15/v30,v14/v29,v13/v28,v12/v27,v11/v26,v10/v25,v9/v24,v8/v23,v7/v22,v6/v21,v16/v18,v17/v19,v18/v20,v19/v21,v20/v22,v21/v23,v22/v24,v23/v25,v24/v26,v25/v27,v26/v28,v27/v29,v28/v30,v29/v16,v30/v17} \draw (\from) -- (\to);
  
  \end{tikzpicture}
\caption{An example for the upper bound of Lemma~\ref{bound_4}: a DDS of $P_{15,2}$.}\label{P_15,2}
\end{center}
\end{figure}

\begin{lem}\label{bound_4}
Let $\Sigma = (P_{n,k}, \sigma)$ be any signed generalized Petersen graph, where $\gcd(n,k) =1$ and $k\geq2$. Let $\lceil \frac{n}{k} \rceil = 2m$, for some $m\geq2$. Then $ n \leq \gamma_{\times 2}(\Sigma) \leq 2n-mk$.
\end{lem}
\begin{proof}
Let $V_{1},V_{2},...,V_{2m}$ be a partition of the set $V_v$ such that $V_{i}= \{v_{(i-1)k},v_{(i-1)k+1},...v_{(i-1)k+(k-1)}\}$ and $V_{2m} = V_v- \cup_{i=1}^{2m-1}V_i$ for $1 \leq i \leq 2m-1$. Note that $|V_{i}| = k$ for each $1 \leq i \leq 2m-1$. Thus we have $|V_{2m}| = n-k(2m-1)$.

Consider the set $D=U \cup \left(\cup_{i=1}^{m} V_{2i}\right)$. Clearly $|D| = n+k(m-1)+n-k(2m-1) = 2n-km$. For example, see Figure~\ref{P_15,2}. We prove that $D$ is a DDS of $\Sigma$, and this will give us the required upper bound.

 Since the cycle $C_o$ lies completely inside $G[D]$, every $u\text{-vertex}$ is dominated at least twice by $D$. Note that for each $v\text{-vertex}$, the corresponding neighbor $u\text{-vertex}$ is in $D$ as $U \subseteq D$. Thus to show that $D$ dominates every vertex of $V_{v}$ at least twice, we just need to show that each vertex of $V_v$ is either in $D$ or adjacent to at least one $v\text{-vertex}$ in $D$. Clearly each vertex of $\cup_{i=1}^{m} V_{2i}$ is in $D$. Further, for $1\leq i \leq m-1$, each vertex of $V_{2i-1}$ is adjacent to a $v\text{-vertex}$ in $D$ because each vertex of $V_{2i-1}$ is adjacent to a vertex of $V_{2i}$ and $V_{2i} \subseteq D$. Also each vertex of $V_{2m-1}$ is adjacent to a vertex of $V_{2m-2}$ and $V_{2m-2} \subseteq D$. Therefore each vertex of $V_{2m-1}$ is also adjacent to a $v\text{-vertex}$ in $D$. Hence $D$ is a DDS of $P_{n,k}$. 
 
 Now it remains to show that $\Sigma[D:V \setminus D]$ is balanced. To do so, it is enough to show that $\Sigma[D:V \setminus D]$ is acyclic. Note that every $u\text{-vertex}$ is adjacent to at most one vertex of $V \setminus D$ since $C_o$ lies completely inside $G[D]$. Therefore if $\Sigma[D:V \setminus D]$ contains any cycle then that cycle must be the inner cycle $C_{i}$ itself. Further, the vertex $v_{0} \in V_{1}$ and the vertex $v_{2mk-k-1} \in V_{2m-1}$. Also both the vertices $v_{0}~\text{and}~v_{2mk-k-1}$ belong to the set $V\setminus D$, and they are adjacent to each other. Therefore $\Sigma[D:V \setminus D]$ is acyclic, and so $\Sigma [D:V \setminus D]$ is balanced. Hence $D$ is a DDS of $\Sigma$. This implies that $\gamma_{\times 2}(\Sigma) \leq 2n-km$.

 The lower bound follows from Theorem~\ref{general bounds}, and the proof is complete.
\end{proof}

\begin{theorem}\label{bound_(n,k)=1}
Let $\Sigma = (P_{n,k}, \sigma)$ be any signed generalized Petersen graph, where $\gcd(n,k) =1$ and $k\geq2$. Then $ n \leq \gamma_{\times 2}(\Sigma) \leq \frac{3n}{2}. $
\end{theorem}
\begin{proof}
For any positive integers $n~\text{and}~k$ it is always true that $\lfloor \frac{n}{k} \rfloor \leq \frac{n}{k} \leq \lceil \frac{n}{k} \rceil$. So for $\lceil \frac{n}{k} \rceil = 2m$, we have $\frac{n}{2} \leq mk$. Therefore the upper bound of Lemma~\ref{bound_4} can be replaced by $\frac{3n}{2}$.

Further, with the assumptions of Lemma~\ref{bound_3}, we have $2m = \lfloor \frac{n}{k} \rfloor \leq \frac{n}{k} \leq \lceil \frac{n}{k} \rceil = 2m+1$ and this implies that $mk \leq \frac{n}{2} $. Thus the upper bound of Lemma~\ref{bound_3} can also be replaced by $\frac{3n}{2}$. Hence we conclude that $ n \leq \gamma_{\times 2}(\Sigma) \leq \frac{3n}{2}$. This completes the proof.
\end{proof}

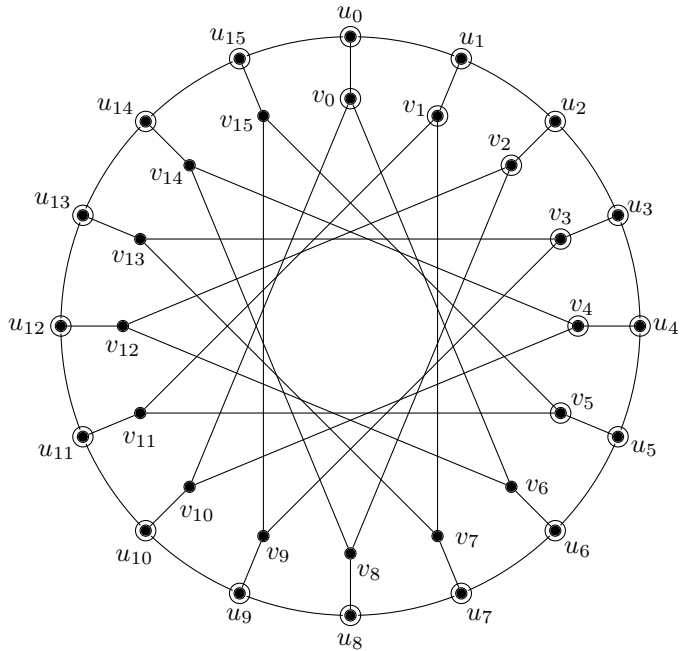
\begin{figure}[h]
\begin{center}
\begin{tikzpicture}[scale=0.55]

		\node[vertex] (v5) at ({360/16*0 }:7cm) {};
		\node[vertex] (v4) at ({360/16*1 }:7cm) {};
		\node[vertex] (v3) at ({360/16 *2 }:7cm) {};
		\node[vertex] (v2) at ({360/16 *3 }:7cm) {};
		\node[vertex] (v1) at ({360/16 *4 }:7cm) {};
		\node[vertex] (v16) at ({360/16 *5}:7cm) {};
		\node[vertex] (v15) at ({360/16 *6 }:7cm) {};
		\node[vertex] (v14) at ({360/16 *7}:7cm) {};
		\node[vertex] (v13) at ({360/16 *8}:7cm) {};
		\node[vertex] (v12) at ({360/16 *9}:7cm) {};
		\node[vertex] (v11) at ({360/16 *10}:7cm) {};
		\node[vertex] (v10) at ({360/16 *11}:7cm) {};
        \node[vertex] (v9) at ({360/16 *12}:7cm) {};
        \node[vertex] (v8) at ({360/16 *13}:7cm) {};
        \node[vertex] (v7) at ({360/16 *14}:7cm) {};
        \node[vertex] (v6) at ({360/16 *15}:7cm) {};
        \draw ({360/16 *0 +1.5}:7cm) arc  ({360/16*0+1.5}:{360/16*1 -1.5 }:7cm);
        \draw ({360/16 *1 +1.5}:7cm) arc  ({360/16*1+1.5}:{360/16*2 -1.5 }:7cm);
        \draw ({360/16 *2 +1.5}:7cm) arc  ({360/16*2+1.5}:{360/16*3 -1.5 }:7cm);
        \draw ({360/16 *3 +1.5}:7cm) arc  ({360/16*3+1.5}:{360/16*4 -1.5 }:7cm);
        \draw ({360/16 *4 +1.5}:7cm) arc  ({360/16*4+1.5}:{360/16*5 -1.5 }:7cm);
        \draw ({360/16 *5 +1.5}:7cm) arc  ({360/16*5+1.5}:{360/16*6 -1.5 }:7cm);
        \draw ({360/16 *6 +1.5}:7cm) arc  ({360/16*6+1.5}:{360/16*7 -1.5 }:7cm);
        \draw ({360/16 *7 +1.5}:7cm) arc  ({360/16*7+1.5}:{360/16*8 -1.5 }:7cm);
        \draw ({360/16 *8 +1.5}:7cm) arc  ({360/16*8+1.5}:{360/16*9 -1.5 }:7cm);
        \draw ({360/16 *9 +1.5}:7cm) arc  ({360/16*9+1.5}:{360/16*10 -1.5 }:7cm);
        \draw ({360/16 *10 +1.5}:7cm) arc  ({360/16*10+1.5}:{360/16*11 -1.5 }:7cm);
        \draw ({360/16 *11+1.5}:7cm) arc  ({360/16*11+1.5}:{360/16*12 -1.5 }:7cm);
        \draw ({360/16 *12 +1.5}:7cm) arc  ({360/16*12+1.5}:{360/16*13 -1.5 }:7cm);
        \draw ({360/16 *13 +1.5}:7cm) arc  ({360/16*13+1.5}:{360/16*14 -1.5 }:7cm);
        \draw ({360/16 *14 +1.5}:7cm) arc  ({360/16*14+1.5}:{360/16*15 -1.5 }:7cm);
        \draw ({360/16 *15 +1.5}:7cm) arc  ({360/16*15+1.5}:{360/16*16 -1.5 }:7cm);
        
        \node[xshift=0.35cm] at ({360/16*0 }:7cm) {$u_{4}$};
		\node[xshift=0.3cm, yshift=0.1cm] at ({360/16*1 }:7cm) {$u_{3}$};
		\node[xshift=0.25cm, yshift=0.2cm] at ({360/16 *2 }:7cm) {$u_{2}$};
		\node[xshift=0.15cm,yshift=0.25cm] at ({360/16 *3 }:7cm) {$u_{1}$};
		\node[ yshift=0.3cm] at ({360/16 *4 }:7cm) {$u_{0}$};
		\node[xshift=-0.15cm, yshift=0.3cm] at ({360/16 *5}:7cm) {$u_{15}$};
		\node[xshift=-0.4cm, yshift=0.2cm] at ({360/16 *6 }:7cm) {$u_{14}$};
		\node[xshift=-0.4cm, yshift=0.2cm] at ({360/16 *7}:7cm) {$u_{13}$};
		\node[xshift=-0.45cm] at ({360/16 *8}:7cm) {$u_{12}$};
		\node[xshift=-0.35cm,  yshift=-0.2cm] at ({360/16 *9}:7cm) {$u_{11}$};
		\node[xshift=-0.15cm, yshift=-0.35cm] at ({360/16 *10}:7cm) {$u_{10}$};
		\node[ yshift=-0.3cm] at ({360/16 *11}:7cm) {$u_{9}$};
        \node[yshift=-0.35cm] at ({360/16 *12}:7cm) {$u_{8}$};
        \node[xshift=0.25cm, yshift=-0.3cm] at ({360/16 *13}:7cm) {$u_{7}$};
  \node[xshift=0.3cm, yshift=-0.25cm] at ({360/16 *14}:7cm) {$u_{6}$};
  \node[xshift=0.35cm, yshift=-0.15cm] at ({360/16 *15}:7cm) {$u_{5}$};
  
  \filldraw[black] ({360/16*0 }:7cm) circle (3pt);
  \filldraw [black] ({360/16 *1}:7cm) circle (3pt);
  \filldraw [black] ({360/16 *2}:7cm) circle (3pt);
  \filldraw [black] ({360/16 *3}:7cm) circle (3pt);
  \filldraw [black] ({360/16 *4}:7cm) circle (3pt);
  \filldraw [black] ({360/16 *5}:7cm) circle (3pt);
  \filldraw [black] ({360/16 *6}:7cm) circle (3pt);
  \filldraw [black] ({360/16 *7}:7cm) circle (3pt);
  \filldraw [black] ({360/16 *8}:7cm) circle (3pt);
  \filldraw [black] ({360/16 *9}:7cm) circle (3pt);
  \filldraw [black] ({360/16 *10}:7cm) circle (3pt);
  \filldraw [black] ({360/16 *11}:7cm) circle (3pt);
  \filldraw [black] ({360/16 *12}:7cm) circle (3pt);
  \filldraw [black] ({360/16 *13}:7cm) circle (3pt);
  \filldraw [black] ({360/16 *14}:7cm) circle (3pt);
  \filldraw [black] ({360/16 *15}:7cm) circle (3pt); 
  
  \draw [black] ({360/16*0 }:7cm) circle (7pt);
  \draw [black] ({360/16 *1}:7cm) circle (7pt);
  \draw [black] ({360/16 *2}:7cm) circle (7pt);
  \draw [black] ({360/16 *3}:7cm) circle (7pt);
  \draw [black] ({360/16 *4}:7cm) circle (7pt);
  \draw [black] ({360/16 *5}:7cm) circle (7pt);
  \draw [black] ({360/16 *6}:7cm) circle (7pt);
  \draw [black] ({360/16 *7}:7cm) circle (7pt);
  \draw [black] ({360/16 *8}:7cm) circle (7pt);
  \draw [black] ({360/16 *9}:7cm) circle (7pt);
  \draw [black] ({360/16 *10}:7cm) circle (7pt);
  \draw [black] ({360/16 *11}:7cm) circle (7pt);
  \draw [black] ({360/16 *12}:7cm) circle (7pt);
  \draw [black] ({360/16 *13}:7cm) circle (7pt);
  \draw [black] ({360/16 *14}:7cm) circle (7pt);
  \draw [black] ({360/16 *15}:7cm) circle (7pt);
    
		\node[vertex] (v21) at ({360/16*0 }:5.5cm) {};
		\node[vertex] (v20) at ({360/16*1 }:5.5cm) {};
		\node[vertex] (v19) at ({360/16 *2 }:5.5cm) {};
		\node[vertex] (v18) at ({360/16 *3 }:5.5cm) {};
		\node[vertex] (v17) at ({360/16 *4 }:5.5cm) {};
		\node[vertex] (v32) at ({360/16 *5}:5.5cm) {};
		\node[vertex] (v31) at ({360/16 *6 }:5.5cm) {};
		\node[vertex] (v30) at ({360/16 *7}:5.5cm) {};
		\node[vertex] (v29) at ({360/16 *8}:5.5cm) {};
		\node[vertex] (v28) at ({360/16 *9}:5.5cm) {};
		\node[vertex] (v27) at ({360/16 *10}:5.5cm) {};
		\node[vertex] (v26) at ({360/16 *11}:5.5cm) {};
        \node[vertex] (v25) at ({360/16 *12}:5.5cm) {};
        \node[vertex] (v24) at ({360/16 *13}:5.5cm) {};
        \node[vertex] (v23) at ({360/16 *14}:5.5cm) {};
        \node[vertex] (v22) at ({360/16 *15}:5.5cm) {};

       \filldraw[black] ({360/16*0 }:5.5cm) circle (3pt);
  \filldraw [black] ({360/16 *1}:5.5cm) circle (3pt);
  \filldraw [black] ({360/16 *2}:5.5cm) circle (3pt);
  \filldraw [black] ({360/16 *3}:5.5cm) circle (3pt);
  \filldraw [black] ({360/16 *4}:5.5cm) circle (3pt);
  \filldraw [black] ({360/16 *5}:5.5cm) circle (3pt);
  \filldraw [black] ({360/16 *6}:5.5cm) circle (3pt);
  \filldraw [black] ({360/16 *7}:5.5cm) circle (3pt);
  \filldraw [black] ({360/16 *8}:5.5cm) circle (3pt);
  \filldraw [black] ({360/16 *9}:5.5cm) circle (3pt);
  \filldraw [black] ({360/16 *10}:5.5cm) circle (3pt);
  \filldraw [black] ({360/16 *11}:5.5cm) circle (3pt);
  \filldraw [black] ({360/16 *12}:5.5cm) circle (3pt);
  \filldraw [black] ({360/16 *13}:5.5cm) circle (3pt);
  \filldraw [black] ({360/16 *14}:5.5cm) circle (3pt);
  \filldraw [black] ({360/16 *15}:5.5cm) circle (3pt);         
        
        \draw ({360/16 *0 +1.5}:5.5cm) ;
        \draw ({360/16 *1 +1.5}:5.5cm) ;
        \draw ({360/16 *2 +1.5}:5.5cm) ;
        \draw ({360/16 *3 +1.5}:5.5cm) ;
        \draw ({360/16 *4 +1.5}:5.5cm) ;
        \draw ({360/16 *5 +1.5}:5.5cm) ;
        \draw ({360/16 *6 +1.5}:5.5cm) ;
        \draw ({360/16 *7 +1.5}:5.5cm) ;
        \draw ({360/16 *8 +1.5}:5.5cm) ;
        \draw ({360/16 *9 +1.5}:5.5cm) ;
        \draw ({360/16 *10 +1.5}:5.5cm) ;
        \draw ({360/16 *11+1.5}:5.5cm) ;
        \draw ({360/16 *12 +1.5}:5.5cm) ;
        \draw ({360/16 *13 +1.5}:5.5cm) ;
        \draw ({360/16 *14 +1.5}:5.5cm) ;
        \draw ({360/16 *15 +1.5}:5.5cm) ;
        
    \draw [black] ({360/16 *0}:5.5cm) circle (7pt);
   \draw [black] ({360/16 *1}:5.5cm) circle (7pt);
   \draw [black] ({360/16 *2}:5.5cm) circle (7pt);
   \draw [black] ({360/16 *3}:5.5cm) circle (7pt);
   \draw [black] ({360/16 *4}:5.5cm) circle (7pt);  
   \draw [black] ({360/16 *15}:5.5cm) circle (7pt);   
        
       \node[xshift=0.05cm, yshift=0.25cm] at ({360/16*0 }:5.5cm) {$v_{4}$};
		\node[ yshift=0.3cm] at ({360/16*1 }:5.5cm) {$v_{3}$};
		\node[xshift=-0.15cm, yshift=0.3cm] at ({360/16 *2 }:5.5cm) {$v_{2}$};
		\node[xshift=-0.3cm, yshift=0.1cm] at ({360/16 *3 }:5.5cm) {$v_{1}$};
		\node[ xshift=-0.35cm] at ({360/16 *4 }:5.5cm) {$v_{0}$};
		\node[xshift=-0.35cm, yshift=-0.1cm] at ({360/16 *5}:5.5cm) {$v_{15}$};
		\node[xshift=-0.3cm, yshift=-0.15cm] at ({360/16 *6 }:5.5cm) {$v_{14}$};
		\node[xshift=-0.15cm, yshift=-0.25cm] at ({360/16 *7}:5.5cm) {$v_{13}$};
		\node[yshift=-0.3cm] at ({360/16 *8}:5.5cm) {$v_{12}$};
		\node[ yshift=-0.35cm] at ({360/16 *9}:5.5cm) {$v_{11}$};
		\node[xshift=0.1cm, yshift=-0.35cm] at ({360/16 *10}:5.5cm) {$v_{10}$};
		\node[xshift=0.2cm, yshift=-0.25cm] at ({360/16 *11}:5.5cm) {$v_{9}$};
        \node[xshift=0.25cm, yshift=-0.25cm] at ({360/16 *12}:5.5cm) {$v_{8}$};
        \node[xshift=0.4cm, yshift=-0.05cm] at ({360/16 *13}:5.5cm) {$v_{7}$};
  \node[xshift=0.35cm] at ({360/16 *14}:5.5cm) {$v_{6}$};
  \node[xshift=0.3cm, yshift=0.15cm] at ({360/16 *15}:5.5cm) {$v_{5}$};    

       \foreach \from/\to in {v1/v17,v2/v18,v3/v19,v4/v20,v5/v21,v6/v22,v7/v23,v8/v24,v9/v25,v10/v26,v11/v27,v12/v28,v13/v29,v14/v30,v15/v31,v16/v32,v17/v23,v18/v24,v19/v25,v20/v26,v21/v27,v22/v28,v23/v29,v24/v30,v25/v31,v26/v32,v27/v17,v28/v18,v29/v19,v30/v20,v31/v21,v32/v22} \draw (\from) -- (\to);
        
	\end{tikzpicture}
\caption{An example for the upper bound of Theorem~\ref{bound_5}: a DDS of $P_{16,6}$.}\label{P_16,6}
\end{center}
\end{figure}

Finally, we give a lower bound and an upper bound for the DDN of signed generalized Petersen graphs, where $\gcd(n,k)=d \geq 2$.

\begin{theorem}\label{bound_5}
Let $\Sigma = (P_{n,k}, \sigma)$ be any signed generalized Petersen graph, where $\gcd(n,k)=d \geq 2$. Then $$n \leq \gamma_{\times 2}(\Sigma) \leq n+d \Big\lceil \frac{n}{3d} \Big\rceil.$$
\end{theorem}
\begin{proof}
Since $\gcd(n,k)=d\geq2$, $P_{n,k}$ has exactly $d$ disjoint $\frac{n}{d}\text{-cycles}$ induced by vertices of $V_{v}$. For each $1 \leq r \leq d$, let $C_{r} = v_{(r-1)}v_{(r-1)+k}v_{(r-1)+2k}...v_{(r-1)+(\frac{n}{d}-1)k}v_{(r-1)}$ be a cycle of length $\frac{n}{d}$. Let $V_{r} = \{v_{(r-1)+3(j-1)k}~|~j = 1,2,...,\lceil \frac{n}{3d} \rceil\} \subset V(C_{r})$. For each $1 \leq r \leq d$, it is clear that $|V_r| = \lceil \frac{n}{3d} \rceil$. Note that every vertex of $V(C_{r}) \setminus V_r$ is adjacent to at least one vertex of $V_r$ for $1 \leq r \leq d$.

  To get the upper bound, consider the set $D = U \cup \left(\cup_{r=1}^{d}V_{r}\right)$. For example, see Figure~\ref{P_16,6}. It is clear that $|D| = n+d \lceil \frac{n}{3d} \rceil$. Also it is clear that $D$ dominates every vertex of $U$ at least twice. Note that every vertex of $V_{v}$ is either in $D$ and has one neighbor in $D$ or in $V\setminus D$ and has two neighbors in $D$. Therefore $D$ is a DDS of $P_{n,k}$.
  
Now we show that $\Sigma[D:V \setminus D]$ is balanced. Note that if $\Sigma[D:V \setminus D]$ contains any cycle then that cycle must be one of the $C_r\text{'s}$, for some $1 \leq r \leq d$, as the outer cycle $C_o$ completely lies inside $G[D]$. Also for each $1 \leq r \leq d$, two consecutive vertices $v_{(r-1)+k}~\text{and}~v_{(r-1)+2k}$ of $C_r$ are contained in $V \setminus D$. This implies that $\Sigma[D:V \setminus D]$ cannot contain a cycle. Hence $\Sigma[D:V \setminus D]$ is acyclic, and so $\Sigma[D:V \setminus D]$ is balanced. Thus we have $\gamma_{\times 2}(\Sigma) \leq n+d\lceil \frac{n}{3d} \rceil$.
  
The lower bound follows from Theorem~\ref{general bounds}, and the proof is complete. 
\end{proof}

\subsection{Bounds on $\gamma_{\times 2}(I(n,j,k),\sigma)$}

It is clear that $P_{n,k} = I(n,1,k)$. Since $I(n,j,k) = I(n,k,j)$ and we wish to get the bounds on $\gamma_{\times2}(I(n,j,k),\sigma)$, we assume that $2\leq j\leq k$. The following theorem gives bounds on $\gamma_{\times2}(I(n,j,k),\sigma)$, for $\gcd(n,k)=1$.

\begin{theorem}\label{bound_I-graph_(n,k)=1}
Let $\Sigma = (I(n,j,k),\sigma)$ be any signed I-graph, where $\gcd(n,k)=1~and~k\geq 2$. Then $$n \leq \gamma_{\times2}(\Sigma) \leq \frac{3n}{2}.$$
\end{theorem}
\begin{proof}
The lower bound follows from Theorem~\ref{general bounds}.

If $\lceil \frac{n}{k} \rceil = 2m$, the set $D$ as considered in Lemma~\ref{bound_4} will be a DDS of $\Sigma$. Therefore $\gamma_{\times2}(\Sigma) \leq 2n-mk$. If $\lceil \frac{n}{k} \rceil = 2m+1$, the set $D$ as considered in Lemma~\ref{bound_3} will be a DDS of $\Sigma$. Therefore $\gamma_{\times2}(\Sigma) \leq n+mk$. To get the required upper bound we mimic the proof of Theorem~\ref{bound_(n,k)=1}. This completes the proof.
\end{proof}

In the following theorem we give bounds on $\gamma_{\times2}(I(n,j,k),\sigma)$, where $\gcd(n,k) \geq 2$.

\begin{theorem}\label{bound_I-graph_(n,k)=d}
Let $\Sigma = (I(n,j,k),\sigma)$ be any signed I-graph, where $\gcd(n,k)=d \geq 2$. Then $$n \leq \gamma_{\times 2}(\Sigma) \leq n+d \Big\lceil \frac{n}{3d} \Big\rceil.$$
\end{theorem}
\begin{proof}
The lower bound follows from Theorem~\ref{general bounds}.

Note that the structure of cycles induced by vertices of $V_{v}$ of $I(n,j,k)$ is same as the structure of cycles induced by vertices of $V_{v}$ of $P_{n,k}$. Since the set $D$ considered in proof of the Theorem~\ref{bound_5} contains the whole set $U$, this same set $D$ will be a DDS of any $\Sigma = (I(n,j,k),\sigma)$. Therefore $$n \leq \gamma_{\times 2}(\Sigma) \leq n+d \Big\lceil \frac{n}{3d} \Big\rceil.$$
This completes the proof.
\end{proof}

\noindent
\textbf{Acknowledgment.} The first author is grateful to Indian Institute of Technology Guwahati, India for providing him a graduate fellowship to carry out the research.

\end{document}